\documentclass[12pt,letterpaper]{article}
\usepackage{amsmath,amsfonts,amsthm,amssymb}

\swapnumbers 
\newtheorem{lemma}{Lemma}[section]
\newtheorem{corollary}[lemma]{Corollary}
\newtheorem{theorem}[lemma]{Theorem}
\newtheorem{example}[lemma]{Example}
\newtheorem{assumptions}[lemma]{Standing Assumptions}

\theoremstyle{definition} 
\newtheorem{definition}[lemma]{Definition}

\newtheorem{remarks}[lemma]{Remarks}

\newcommand{\Nat}{{\mathbb N}}

\newcommand\reals{{\mathbb R}}

\newcommand{\cb}{(J\sb{p})\sb{p\in P}}

\newcommand{\sgn}{\operatorname{sgn}}
\newcommand{\dg}{\sp{\text{\rm o}}}
\newcommand{\opspec}{\operatorname{spec}}

\begin{document}

\title{Synaptic Algebras}

\author{David J. Foulis{\footnote{Emeritus Professor of 
Mathematics and Statistics, University of Massachusetts, 
Postal Address: 1 Sutton Court, Amherst, MA 01002, USA; 
foulis@math.umass.edu.}}} 

\date{}
\maketitle

\begin{center}
Dedicated to Dr. Sylvia Pulmannov\'{a} on the occasion of her 
70$\sp{{\rm th}}$ birthday.
\end{center}

\begin{abstract}
A synaptic algebra is both a special Jordan algebra and a spectral 
order-unit normed space satisfying certain natural conditions 
suggested by the partially ordered Jordan algebra of bounded 
Hermitian operators on a Hilbert space. The adjective ``synaptic," 
borrowed from biology, is meant to suggest that such an algebra  
coherently ``ties together" the notions of a Jordan algebra, a 
spectral order-unit normed space, a convex effect algebra, and 
an orthomodular lattice. Prototypic examples of synaptic algebras 
are the special Jordan algebra of all self-adjoint elements in 
a von Neumann algebra, the self-adjoint elements in a Rickart 
C*-algebra, the self-adjoint elements in an AW*-algebra, D. Topping's 
JW- and AJW-algebras, and the generalized Hermitian (GH-) algebras 
introduced and studied by the author and S. Pulmannov\'{a}. All 
the foregoing examples are norm complete, but synaptic algebras are 
more general, and even a commutative synaptic algebra need not be 
norm complete.
\end{abstract}


\section{Introduction} \label{sc:Intro}  

Our purpose in this article is introduce and study a class 
of partially ordered algebraic structures, which we call \emph{synaptic 
algebras}, that are simultaneously spectral order-unit normed spaces 
\cite{SOUS} and special Jordan algebras \cite{McC}, and that also 
incorporate convex effect algebras \cite{GPBB} and orthomodular lattices 
\cite{Beran, Kalm}. We have borrowed from biology the adjective 
`synaptic,' which is derived from the Greek word `\emph{sunaptein},' 
meaning \emph{to join together}. A synaptic algebra (Definition 
\ref{df:SynapticAlgebra} below) is required to satisfy certain natural 
conditions suggested by an important spacial case, namely the partially 
ordered Jordan algebra of bounded Hermitian operators on a Hilbert space.

The generalized Hermitian (GH) algebras introduced and studied by 
Sylvia Pulmannov\'{a} and the author in \cite{GHAlg1, GHAlg2} are 
synaptic algebras that satisfy a rather strong additional condition 
on bounded ascending sequences of pairwise commuting elements. In  
a synaptic algebra, this condition is replaced by some of its 
important algebraic consequences together with a much weaker condition 
on ascending sequences---see Section \ref{sc:SGH} below for the details. 
Example \ref{ex:functions} below exhibits a commutative synaptic 
algebra which, in general, fails to be a GH-algebra, showing that 
synaptic algebras are proper generalizations of GH-algebras. 

Among other properties of a synaptic algebra $A$, we show that 
the set of projections $P$ in $A$ is the extreme boundary of 
the convex set $E$ of effects in $A$ (Theorem \ref{th:CharacterizeP}), 
$P$ is an orthomodular lattice (Theorem \ref{th:PisanOML}), an 
element in $A$ is (von Neumann) regular iff both its positive and 
negative parts are regular (Theorem \ref{th:RegPosNeg}), each  
element in $A$ determines and is determined by a spectral resolution 
in $P$ (Theorems \ref{th:SpecProps} and \ref{th:SpecTh}), and 
each element in $A$ is a norm limit of an ascending sequence of 
simple elements (Corollary \ref{co:AscendingLimit}). 

In the sequel, we use the symbols $\reals$ and $\Nat$ for the 
ordered field of real numbers and the set of positive integers, 
respectively. Also, we use `iff' as an abbreviation for `if and 
only if,' and the symbol `$:=$' means `equals by definition.' 
To help fix ideas in the following definition, the reader can 
keep in mind the special case in which $R$ is the algebra of all 
bounded linear operators on a Hilbert space, and $A$ is the real 
vector space of all self-adjoint operators in $R$.  

\begin{definition} \label{df:SynapticAlgebra}
Let $R$ be a linear associative algebra with unity element $1$ 
over $\reals$ and let $A$ be a (real) vector subspace of $R$.  
If $a,b\in A$ and $B\subseteq A$, we write $aCb$ iff $a$ and $b$ 
commute (i.e. $ab=ba$)\footnote{We understand that a product of 
elements of $A$ is the product as calculated in $R$.} and we 
define $$C(a) :=\{b\in A:aCb\},\ C(B) :=\bigcap\sb{b\in B}C(b),  
\text{\ and\ }CC(a) :=C(C(a)).$$ The vector space $A$ is a \emph{synaptic 
algebra} with \emph{enveloping algebra} $R$ iff the following 
conditions are satisfied:  
\begin{enumerate}
\item[SA1.] $A$ is a partially ordered archimedean real vector space 
 with positive cone $A\sp{+}=\{a\in A:0\leq a\}$, $1\in A\sp{+}$ is 
 an order unit in $A$, and $\|\cdot\|$ is the corresponding order-unit 
 norm.\footnote{See Definition \ref{df:OUNormSpace} below.}
\item[SA2.] If $a,b\in A\sp{+}$, then $aCb\Rightarrow ab\in A\sp{+}$.
\item[SA3.] If $a\in A$ then $a\sp{2}\in A\sp{+}$.
\item[SA4.] If $a,b\in A\sp{+}$, then $aba\in A\sp{+}$.
\item[SA5.] If $a\in A$ and $b\in A\sp{+}$, then $aba=0\Rightarrow 
 ab=ba=0$.
\item[SA6.] If $a\in A\sp{+}$, there exists $b\in A\sp{+}\cap CC(a)$ 
 such that $b\sp{2}=a$.
\item[SA7.] If $a\in A$, there exists $p\in A$ such that $p=p\sp{2}$ and, 
 for all $b\in A$, $ab=0\Leftrightarrow pb=0$.
\item[SA8.] If $1\leq a\in A$, there exists $b\in A$ such that $ab=ba=1$.
\item[SA9.] If $a,b\in A$, $a\sb{1}\leq a\sb{2}\leq a\sb{3}\leq\cdots$ 
 is an  ascending sequence of pairwise commuting elements of $C(b)$ 
 and $\lim\sb{n\rightarrow\infty}\|a-a\sb{n}\|=0$, then $a\in C(b)$.
\end{enumerate} 
We define $P :=\{p\in A:p=p\sp{2}\}$. Elements $p\in P$ are 
called \emph{projections}. We define the \emph{unit interval} 
$E$ in $A$ by $E :=\{e\in A:0\leq e\leq 1\}$. Elements $e\in E$ 
are called \emph{effects}.\footnote{Actually, $E$ is a so-called 
\emph{convex effect algebra} \cite{GPBB}.}
\end{definition}

If $R$ is a von Neumann algebra, then the real vector space $A$ of 
all self-adjoint elements in $R$ is a synaptic algebra. More generally, 
the self-adjoint elements in a Rickart C*-algebra \cite[\S 3]{HHL}, 
and in particular in an AW*-algebra \cite{Kap}, form a synaptic algebra. 
Additional examples of synaptic algebras are: JW-algebras \cite{Top65}, 
AJW-algebras \cite[\S\,20]{Top65}, JB-algebras \cite{ASS78}, and the 
ordered special Jordan algebras studied by Sarymsakov, \emph{et al.} 
\cite{Sary}. All the foregoing examples are norm complete, but the 
commutative synaptic algebra in the following example need not be 
norm complete.

\begin{example} \label{ex:functions}
Let ${\mathcal F}$ be a field of subsets of a nonempty set $X$, 
let $A$ be the commutative and associative real linear algebra, 
with pointwise operations, of all functions $f\colon X\to\reals$
such that {\rm (i)} $\lambda\in\reals\Rightarrow f\sp{-1}(\lambda)
\in{\mathcal F}$ and {\rm (ii)} $\{f(x):x\in X\}$ is finite. Then, 
with the pointwise partial order, $A$ is a synaptic algebra with 
$A$ as its own enveloping algebra. The projections in $A$ are the 
characteristic set functions {\rm (}indicator functions{\rm 
)} of sets in ${\mathcal F}$.
\end{example}

\begin{assumptions}
In the sequel, we assume that $A$ is a synaptic algebra with 
enveloping algebra\,\footnote{We shall not be concerned with the 
detailed structure of the enveloping algebra $R$---we regard $R$ 
merely as an arena in which to study $A$, $E$, and $P$.} $R$, that 
$E$ is the set of effects in $A$, and that $P$ is the set of 
projections in $A$. We understand that both $E$ and $P$ are 
partially ordered by the restrictions  of the partial order 
$\leq$ on $A$. To avoid triviality, we assume that $1\not=0$. As 
is customary, we shall identify each real number $\lambda\in
\reals$ with the element $\lambda1\in A$, so that $\reals$ is a 
one-dimensional linear subspace of $A$. If $n$ is one of 
$1,2,...,9$, then {\rm [SA}$n${\rm ]} will always refer to 
the corresponding condition in Definition {\rm \ref{df:SynapticAlgebra}}. 
\end{assumptions}

By [SA1], $A$ is an \emph{order-unit normed space} according 
to the following definition (adapted to our present notation).

\begin{definition} \label{df:OUNormSpace}
An \emph{order-unit normed space} \cite[pp. 67--69]{Alf}  is a 
partially ordered real vector space $A$ with a distinguished element 
$1\in A$, called the \emph{unit}, such that:
\begin{enumerate}
\item $A$ is \emph{archimedean}, i.e., if $a,b\in A$ and $na\leq b$ 
 for all $n\in\Nat$, then $-a\in A\sp{+}$.
\item $0<1$ and $1$ is an \emph{order unit}\,\footnote{Some authors use 
 the terminology ``strong order unit."} in $A$, i.e., for every $a\in A$,  
 there exists $n\in\Nat$ such that $a\leq n$.\footnote{Recall that we are 
 identifying $n\in\Nat\subseteq\reals$ with $n1$.}
\end{enumerate}
The \emph{order-unit norm} $\|\cdot\|$ on $A$ is defined by 
\begin{enumerate}
\setcounter{enumi}{2}
\item $\|a\| :=\inf\{\lambda\in\reals:0<\lambda\text{\ and\ }
 -\lambda\leq a\leq\lambda\}$.
\end{enumerate}
\end{definition}

The order-unit norm $\|\cdot\|$ is a \emph{bona fide} norm on $A$, 
and it is related to the partial-order structure of $A$ by the 
following properties\footnote{See \cite[Proposition II.1.2]{Alf} 
and \cite[Proposition 7.12 (c)]{Good}}, which we shall use routinely 
in the sequel: For all $a,b\in A$, 
\[
-\|a\|\leq a\leq\|a\|,\text{\ and if\ }-b\leq a\leq b,\text{\ then\ }
 \|a\|\leq\|b\|.
\]
\emph{If $(a\sb{n})\sb{n\in\Nat}$ is a sequence in $A$ and $a\in A$, 
the notation $\lim\sb{n\rightarrow\infty}a\sb{n}=a$, or simply 
$a\sb{n}\rightarrow a$, will mean that $a$ is the limit of 
$(a\sb{n})\sb{n\in\Nat}$ in the norm topology, i.e., that}  
$\lim\sb{n\rightarrow\infty}\|a-a\sb{n}\|=0$.

By [SA3], $a\in A\Rightarrow a\sp{2}\in A$, hence $A$ is organized 
into a special Jordan algebra \cite{McC} under the Jordan product 
$$a\circ b :=\frac12(ab+ba)=\frac12[(a+b)\sp{2}-a\sp{2}-b\sp{2}]
\in A\text{\ for all\ }a,b\in A.$$ Clearly, $1\circ a=a\circ 1=a$, 
i.e., $A$ is a \emph{unital} Jordan algebra.

\begin{remarks} \label{rm:Jordan} 
Let $a,b,c\in A$. Then $aCb\Rightarrow ab=ba=a\circ b\in A$. As 
$a\sp{2}\in A$ and $aCa\sp{2}$, it follows that $a\sp{3}=a\circ 
a\sp{2}\in A$, and by induction, $a\sp{n}\in A$ for all $n\in\Nat$. 
Consequently, $A$ is closed under the formation of real polynomials 
in $a$. Let $c :=2(a\circ b)$. Then $aba=a\circ c-a\sp{2}\circ b
\in A$, hence $aba\in A$. Thus, $acb+bca=(a+b)c(a+b)-aca-bcb\in A$. 
\end{remarks}

\begin{lemma} \label{lm:Norm}
Let $a,b\in A$ and $0<\lambda\in\reals$. Then:
{\rm (i)} $-\lambda\leq a\leq \lambda\Leftrightarrow a\sp{2}
\leq\lambda\sp{2}$. {\rm (ii)} $\|a\sp{2}\|=\|a\|\sp{2}$. 
{\rm (iii)} $0\leq a,b\Rightarrow\|a-b\|\leq\max\{\|a\|,\|b\|\}$.  
{\rm (iv)} $\|a\circ b\|\leq\|a\|\|b\|$. {\rm (v)} If $aCb$, 
then $\|ab\|\leq\|a\|\|b\|$. 
\end{lemma}

\begin{proof}
If $-\lambda\leq a\leq \lambda$, then $0\leq\lambda-a,\lambda+a$, 
and as $(\lambda-a)C(\lambda+a)$, [SA2] implies that $0\leq
(\lambda-a)(\lambda+a)=\lambda\sp{2}-a\sp{2}$. Conversely, suppose 
that $a\sp{2}\leq\lambda\sp{2}$. Then $0\leq(\lambda-a)\sp{2}$ by 
[SA3], whence $0\leq(\lambda\sp{2}-a\sp{2})+(\lambda-a)\sp{2}=
2(\lambda\sp{2}-\lambda a)$ and since $0<\lambda$, it follows that 
$a\leq\lambda$. As $a\sp{2}\leq\lambda\sp{2}$, we also have $(-a)
\sp{2}\leq\lambda\sp{2}$, whence $-a\leq\lambda$, i.e., $-\lambda
\leq a$, proving (i).  Part (ii) follows from (i). To prove (iii), 
we can assume that $\|a\|\leq\|b\|$. As $0\leq b$, we have $a\leq
\|a\|\leq\|b\|\leq\|b\|+b$, whence $a-b\leq\|b\|$. Also, as $0
\leq a$, we have $b\leq\|b\|\leq\|b\|+a$, whence $b-a\leq\|b\|$, 
and therefore $-\|b\|\leq a-b\leq\|b\|$. Consequently, $\|a-b\|
\leq\|b\|=\max\{\|a\|,\|b\|\}$. To prove (iv), it will be sufficient 
by normalization to prove that $\|a\|=\|b\|=1\Rightarrow\|a\circ b\|
\leq 1$. Thus, we assume $\|a\|=\|b\|=1$, so that $\|a\pm b\|
\leq 2$, and therefore by (ii), $\|(a\pm b)\sp{2}\|\leq 4$. 
Consequently, by (iii),
\[
\|a\circ b\|=\frac{1}{4}\|(a+b)\sp{2}-(a-b)\sp{2}\|\leq
 \frac{1}{4}\max\{\|(a+b)\sp{2}\|,\|(a-b)\sp{2}\|\}\leq 1.
\]
If $aCb$, then $ab=a\circ b$, so (v) follows immediately from (iv).
\end{proof}

\section{Square Roots, Projections, and Carriers} \label{sc:SRAB} 

\begin{remarks} \label{rm:NoNilpotents}
Let $a,b\in A$. Then: (i) By [SA5] with $b=1$, we have $a\sp{2}=0
\Rightarrow a=0$. (ii) If $0\leq a,b$ and $a+b=0$, then $0\leq a=
-b\leq 0$, whence $a=b=0$.
\end{remarks}

\begin{theorem} \label{th:UniqueSR}
Let $0\leq a\in A$. Then there exists a unique $r\in A$ such 
that $0\leq r$ and $r\sp{2}=a$; moreover, $r\in CC(a)$.
\end{theorem}

\begin{proof}
Suppose that $0\leq a\in A$. By [SA6], there exists $b\in CC(a)$ 
such that $0\leq b$ and $b\sp{2}=a$. As $a\in C(a)$, we have $aCb$. 
Suppose also that $r\in A$ with $0\leq r$, $r\sp{2}=a$. Obviously, 
$rCa$, whence $bCr$. It will be sufficient to prove that $r=b$. 

By [SA6], there exists $s\in CC(b)$ such that $0\leq s$ and 
$s\sp{2}=b$. As $b,r\in C(b)$, we have $sCb$ and $sCr$. By [SA6] 
again, there exists $t\in CC(r)$ such that $0\leq t$ and $t\sp{2}=r$.
As $b,r\in C(r)$, we have $tCb$ and $tCr$. 

Since $sCb$ and $sCr$, it follows that $sC(b-r)$, hence $s(b-r)=
s\circ(b-r)\in A$. Likewise, since $tCb$ and $tCr$, we have 
$t(b-r)\in A$.  Moreover, as $b\sp{2}=r\sp{2}=a$, it follows that
\[
(s(b-r))\sp{2}+(t(b-r))\sp{2}=(s\sp{2}+t\sp{2})(b-r)\sp{2}
=(b+r)(b-r)\sp{2}=(b\sp{2}-r\sp{2})(b-r)=0.
\]
But $0\leq (s(b-r))\sp{2}$ and $0\leq(t(b-r))\sp{2}$ by [SA3], 
whence $(s(b-r))\sp{2}=(t(b-r))\sp{2}=0$, so $s(b-r)=t(b-r)=0$ by 
Remarks \ref{rm:NoNilpotents}.

As $s(b-r)=0$, it follows that $b(b-r)=s\sp{2}(b-r)=0$. Likewise, 
$r(b-r)=t\sp{2}(b-r)=0$, whence $(b-r)\sp{2}=b(b-r)-r(b-r)=0$, 
and by Remarks \ref{rm:NoNilpotents} (i), $r=b$.
\end{proof}

If $0\leq a\in A$, then of course, the unique element $r$ in Theorem  
\ref{th:UniqueSR} is called the \emph{square root} of $a$, and 
in what follows \emph{we denote it in the usual way as $a\sp{1/2}$.}

\begin{remarks} \label{rm:PsubsetofE}
Let $p\in P$. Then, as $p=p\sp{2}$, [SA3] implies that 
$0\leq p$. Also, $(1-p)\sp{2}=1-2p+p\sp{2}=1-p$, so $1-p\in P$, 
and therefore $0\leq 1-p$, i.e., $p\leq 1$. Consequently, 
$0\leq p\leq 1$, and it follows that $P\subseteq E$.
\end{remarks}   

\begin{theorem} \label{th:e<=p}
Let $e\in E$ and $p\in P$. Then the following conditions are 
mutually equivalent: {\rm (i)} $e\leq p$. {\rm (ii)} $e=ep=pe$. 
{\rm (iii)} $e=pep$. {\rm (iv)} $e=ep$. {\rm (v)} $e=pe$.
\end{theorem}

\begin{proof}
(i) $\Rightarrow$ (ii). Assume that $e\leq p$ and let $d :=p-e$. 
Then $0\leq e,d,1-p$, $e+d=p$, and 
\[
(1-p)e(1-p)+(1-p)d(1-p)=(1-p)p(1-p)=0.
\]
By [SA4], $0\leq (1-p)e(1-p),\,(1-p)d(1-p)$, and it follows from 
Remarks \ref{rm:NoNilpotents} (ii) that $(1-p)e(1-p)=(1-p)d(1-p)=0$. 
Therefore, by [SA5], $(1-p)e=e(1-p)=0$, i.e., $e=pe=ep$.

\smallskip 

(ii) $\Rightarrow$ (iii) $\Rightarrow$ (iv). Follows from 
$p=p\sp{2}$.

\smallskip

(iv) $\Leftrightarrow$ (v). By [SA5],
$e=ep\Rightarrow e(1-p)=0\Rightarrow (1-p)e(1-p)=0\Rightarrow(1-p)e=0
\Rightarrow e=pe$, and the converse implication follows by symmetry.

\smallskip

(v) $\Rightarrow$ (i).  Assume (v). Since (iv) 
$\Leftrightarrow$ (v), we have $pe=ep=e$, so $(1-e)p=p(1-e)
=p-e$, whence  $0\leq p-e$ by [SA2], and therefore $e\leq p$.
\end{proof}

\begin{lemma} \label{lm:eSquared}
Let $e\in E$. Then: {\rm (i)} $e\sp{2}\in E$ with $0\leq e\sp{2}\leq e$. 
{\rm (ii)} $2e-e\sp{2}\in E$. {\rm (iii)} $e-e\sp{2}\in E$ with $e-e
\sp{2}\leq e,\,1-e$. 
\end{lemma}

\begin{proof}
By [SA3], $0\leq e\sp{2}$, and as $eC(1-e)$ with $0\leq e,1-e$, 
[SA2] implies that $0\leq e(1-e)$, whence $0\leq e\sp{2}\leq e
\leq 1$, proving (i). Also, $0\leq(1-e)\sp{2}=1-2e+e\sp{2}$, so 
by (i), $0\leq e+(e-e\sp{2})=2e-e\sp{2}\leq 1$, proving (ii). 
Part (iii) follows from (i) and (ii). 
\end{proof}

Obviously, $E$ is a convex set, and by Remarks \ref{rm:PsubsetofE}, 
$P\subseteq E$. The following theorem characterizes, in various 
ways, those effects $p\in E$ that are projections.

\begin{theorem}  \label{th:CharacterizeP}
If $p\in E$, then the following conditions are mutually equivalent:
\begin{enumerate}
\item $p\in P$.
\item If $\lambda\in\reals$, $0<\lambda<1$, and $e\in E$, then 
 $\lambda e\leq p\Leftrightarrow e\leq p$.
\item $p$ is an extreme point of the convex set $E$.
\item If $e,f,e+f\in E$, then $e,f\leq p\Rightarrow e+f\leq p$.
\item If $e\in E$ and $e\leq p,1-p$, then $e=0$. 
\end{enumerate}
\end{theorem}

\begin{proof}
(i) $\Rightarrow$ (ii). Suppose $p\in P$, $e\in E$, and 
$0<\lambda<1$. Then $0\leq\lambda e\leq e\leq 1$, so $\lambda e
\in E$. Therefore, by Theorem \ref{th:e<=p}, $\lambda e\leq p
\Leftrightarrow \lambda ep=\lambda e\Leftrightarrow ep=e
\Leftrightarrow e\leq p$.

\smallskip

(ii) $\Rightarrow$ (iii) Assume (ii) and suppose that $p=
\lambda e+(1-\lambda)f$ with $0<\lambda<1$ and $e,f\in E$. 
Then $\lambda e\leq p$, whence $e\leq p=\lambda e+
(1-\lambda)f$, therefore $(1-\lambda)e\leq(1-\lambda)f$, 
and it follows that $e\leq f$. Similarly, $f\leq e$, 
so $e=f=p$.

\smallskip

(iii) $\Rightarrow$ (i) Assume (iii). By parts (i) and (ii) 
of Lemma \ref{lm:eSquared}, $p\sp{2},\,2p-p\sp{2}\in E$, 
and since $p=\frac12p\sp{2}+\frac12(2p-p\sp{2})$, (iii) 
implies that $p=p\sp{2}=2p-p\sp{2}$, whence $p\in P$.

\smallskip

(i) $\Rightarrow$ (iv) Assume that $p\in P$, $e,f,e+f\in E$, 
and $e,f\leq p$. Then by Theorem \ref{th:e<=p}, $e=pep$ and 
$f=pfp$. As $e+f\in E$, we have $0\leq 1-(e+f)$, whence 
by [SA4], $0\leq p(1-e-f)p$, i.e., $e+f=pep+pfp\leq p\sp{2}=p$.

\smallskip

(iv) $\Rightarrow$ (v) Assume (iv) and suppose that 
$e\in E$ with $e\leq p,1-p$. Then $e,p\in E$, $0\leq e+p\leq 1$, 
and $e,p\leq p$, whence $e+p\leq p$ by (iv), and therefore 
$e\leq 0$. But $0\leq e$, so $e=0$.

\smallskip

(v) $\Rightarrow$ (i) Assume (v). By Lemma \ref{lm:eSquared} (iii), 
$0\leq p-p\sp{2}\leq p,1-p$, whence $p=p\sp{2}$ by (v).
\end{proof}

\begin{theorem} \label{lm:UniqueCarrier}
Let $a\in A$. Then there exists a unique projection $p\in P$ such 
that, for all $b\in A$, $ab=0\Leftrightarrow pb=0$. 
\end{theorem}

\begin{proof}
By [SA7], there exists $p\in P$ such that, for all $b\in A$, 
$ab=0\Leftrightarrow pb=0$. Suppose $q\in P$ and, for all $b\in A$, 
$ab=0\Leftrightarrow qb=0$. Putting $b=1-p$, we find that $a(1-p)
=0$, whence $q(1-p)=0$, i.e., $q=qp$, and therefore $q\leq p$ by 
Theorem \ref{th:e<=p}. By symmetry, $p\leq q$, so $p=q$, proving 
the uniqueness of $p$.
\end{proof}

\begin{definition} \label{df:carrier}
If $a\in A$, then the unique projection $p$ in Theorem 
\ref{lm:UniqueCarrier} is called the \emph{carrier projection} 
of (or for) $a$ and is denoted by $a\dg$. Thus, $a\dg\in P$ and, 
for all $b\in A$, $ab=0\Leftrightarrow a\dg b=0$.
\end{definition}

\begin{lemma} \label{lm:ab=0}
Let $a,b\in A$ and $p\in P$. Then: {\rm (i)} $pb=0\Leftrightarrow bp=
0$. {\rm (ii)} $pa=a\Leftrightarrow ap=a$. {\rm (iii)} $aa\dg=a\dg a=
a$. {\rm (iv)} $ab=0\Leftrightarrow ba=0$.
\end{lemma}

\begin{proof}
By [SA5] and the fact that $0\leq p$, we have $pb=0\Rightarrow bpb=0
\Rightarrow bp=0$, whence $pb=0\Rightarrow bp=0$. A similar argument 
yields the converse, proving (i). By (i), $pa=a\Leftrightarrow(1-p)a
=0\Leftrightarrow a(1-p)=0\Leftrightarrow ap=a$, proving (ii). As $a
\dg\in P$, we have $a\dg(1-a\dg)=0$, so $a(1-a\dg)=0$, i.e., $aa\dg=
a$, whence $a\dg a=a$ by (ii), proving (iii). To prove (iv), assume 
that $ab=0$. Then $a\dg b=0$, so $ba\dg=0$ by (i). Also, $a=a\dg a$ 
by (iii), whereupon $ba=ba\dg a=0$. Thus, $ab=0\Rightarrow ba=0$, 
and the converse follows by symmetry.
\end{proof}

\begin{theorem} \label{th:Carrier}
Let $a,b\in A$. Then: {\rm (i)} $a=0\Leftrightarrow a\dg=0$. 
{\rm (ii)} $a\in P\Leftrightarrow a=a\dg$. {\rm (iii)} $a\dg$ is 
the smallest projection $p\in P$ such that $a=ap$. {\rm (iv)} If 
$e\in E$, then $e\dg$ is the smallest projection $p\in P$ such that 
$e\leq p$. {\rm (v)} $ab=0\Leftrightarrow ab\dg=0\Leftrightarrow a
\dg b\dg=0$. {\rm (vi)} $a\dg\in CC(a)$. {\rm (vii)} 
If $n\in\Nat$, then $(a\sp{n})\dg=a\dg$.  {\rm (viii)} If 
$0\leq a\leq b$, then $a\dg\leq b\dg$. 
\end{theorem} 

\begin{proof}
(i) and (ii) are obvious from the definition of $a\dg$.

\smallskip

(iii) We have $aa\dg=a$ by Lemma \ref{lm:ab=0} (iii). Suppose that 
$p\in P$ and $a=ap$. Then $a(1-p)=0$, whence $a\dg(1-p)=0$, so 
$a\dg=a\dg p$, and therefore $a\dg\leq p$ by Theorem \ref{th:e<=p}.

\smallskip

(iv) Part (iv) is a consequence of (iii) and Theorem \ref{th:e<=p}.  

\smallskip

(v) By Lemma \ref{lm:ab=0} (iv), 
\[
ab=0\Leftrightarrow ba=0\Leftrightarrow b\dg a=0\Leftrightarrow ab
 \dg=0\Leftrightarrow a\dg b\dg=0.
\] 

(vi) Suppose that $c\in C(a)$ and let $d :=(1-a\dg)ca
\dg+a\dg c(1-a\dg)$. Thus, $d\in A$ (see Remarks \ref{rm:Jordan}), 
and as $aa\dg=a$, we have 
\[ 
ad=a(1-a\dg)ca\dg+aa\dg c(1-a\dg)=0+ac(1-a\dg)=ca(1-a\dg)=0,
\]
and therefore
\[
0=a\dg d=0+a\dg c(1-a\dg)=a\dg c-a\dg ca\dg,\text{\  i.e.,\ }
 a\dg c=a\dg ca\dg.
\] 
Also, as $a\dg d=0$, Lemma \ref{lm:ab=0} implies that 
$0=da\dg=(1-a\dg)ca\dg$, i.e., $ca\dg=a\dg ca\dg$. Therefore 
$ca\dg=a\dg ca\dg=a\dg c$, so $c\in C(a\dg)$. 

\smallskip

(vii) Let $n\in\Nat$. As $aa\dg=a$, we have $a\sp{n}a\dg=
a\sp{n}$, whence $(a\sp{n})\dg\leq a\dg$ by (iii). We have to prove 
that $a\dg\leq(a\sp{n})\dg$. Put $q :=1-(a\sp{n})\dg$. By (vi), 
$C(a\sp{n})\subseteq C((a\sp{n})\dg)$, whence $aCq$. Evidently,
$a\sp{n}q=0$, so there is a smallest positive integer $k$ such 
that $a\sp{k}q=0$. If $k$ is even, then $a\sp{k/2}qa\sp{k/2}=0$, 
so $a\sp{k/2}q=0$ by [SA5], contradicting the minimality 
of $k$. Therefore, $k$ is odd and $a\sp{k+1}q=0$, whence 
$a\sp{(k+1)/2}qa\sp{(k+1)/2}=0$, so $a\sp{(k+1)/2}q=0$ by 
[SA5] again, whereupon $k\leq(k+1)/2$, i.e., $k=1$. Therefore, 
$aq=0$, whence $a=a(a\sp{n})\dg$, and again by (iii), 
$a\dg\leq(a\sp{n})\dg$.

\smallskip

(viii) Suppose that $0\leq a\leq b$. The case $b=0$ is 
trivial, so we assume that $b\not=0$. Let $\lambda :=
\|b\|\sp{-1}$, $e :=\lambda a$, and $f :=\lambda b$. Clearly, 
$e,f\in E$, $e\leq f$, $e\dg=a\dg$, and $f\dg=b\dg$. By (iv), 
$e\leq f\leq f\dg\in P$, whence $e\dg\leq f\dg$, i.e., $a\dg
\leq b\dg$.
\end{proof}

\section{Absolute Value and Polar Decomposition} 
\label{sc:AVandPD} 

If $a\in A$, then by [SA3], $0\leq a\sp{2}$, so we can 
formulate the following definition.

\begin{definition}
If $a\in A$, then the \emph{absolute value} of $a$ is defined 
and denoted by $|a| :=(a\sp{2})\sp{1/2}$. Also we define 
$a\sp{+} :=\frac12(|a|+a)$ and $a\sp{-} :=\frac12(|a|-a)$. 
\end{definition}

\begin{remarks} \label{rm:AbsProps}
Let $a\in A$. Obviously, $0\leq|a|=|-a|$ and $|a|\sp{2}=a\sp{2}$. 
Also, $C(a)\subseteq C(a\sp{2})\subseteq C(|a|)$, and therefore 
$|a|,a\sp{+},a\sp{-}\in CC(a)$. Moreover, $a=a\sp{+}-a\sp{-}$, 
$|a|=a\sp{+}+a\sp{-}$, $a\sp{+}a\sp{-}=a\sp{-}a\sp{+}=0$, and 
$a\sp{-}=(-a)\sp{+}$. 
\end{remarks}

\begin{theorem} \label{th:Prop a+a-}
Let $a\in G$, $p :=(a\sp{+})\dg$, and $q :=(a\sp{-})\dg$.  
Then:

\smallskip

\begin{tabular}{ll}
\ \ {\rm (i)}\ $p,q\in CC(a)$ &
\ \ \ \ \ \ {\rm (ii)} $pC|a|$ and $qC|a|$\\
\,{\rm (iii)}\ $pa=ap=a\sp{+}$. &
\ \ \ \ \ \,{\rm (iv)}\ $qa=aq=-a\sp{-}$.\\
\ \,{\rm (v)}\ $0\leq p|a|=|a|p=a\sp{+}$. & \ \ \ \ \ 
{\rm (vi)}\ $0\leq q|a|=|a|q=a\sp{-}$.\\
{\rm (vii)} $pq=qp=0$. & \ \ \ \,{\rm (viii)} $p+q=a\dg$.
\end{tabular}
\end{theorem}

\begin{proof}
(i) As $C(a)\subseteq C(a\sp{+})$ and $C(a\sp{+})\subseteq 
C((a\sp{+})\dg)$, we have $C(a)\subseteq C(p)$. Likewise, as 
$a\sp{-}=(-a)\sp{+}$ and $C(a)=C(-a)$, we have $C(a)\subseteq 
C(q)$. 

\smallskip

(ii) As $|a|\in C(a)$, (ii) follows from (i).

\smallskip

(iii) By (i), $pa=ap$. Also, $a\sp{+}=(a\sp{+})\dg a\sp{+}=
pa\sp{+}$, and since $a\sp{+}a\sp{-}=0$, it follows that $pa\sp{-}=0$, 
whence $pa=p(a\sp{+}-a\sp{-})=a\sp{+}$.

\smallskip

(iv) By (iii), $-qa=q(-a)=(-a)\sp{+}=a\sp{-}$.

\smallskip

(v) By (ii), $pC|a|$, and as in the proof of (iii), 
$p|a|=p(a\sp{+}+a\sp{-})=a\sp{+}$. As $0\leq p$ and $0\leq
|a|$, [SA2] implies that $0\leq p|a|$. 

\smallskip

(vi) The proof of (vi) is similar to the proof of (v).

\smallskip

(vii) As $a\sp{+}a\sp{-}=0$, we have $pa\sp{-}=0$, 
whence $pq=0$.

\smallskip

(viii) By (vii), $(p+q)\sp{2}=p\sp{2}+q\sp{2}=p+q$, 
so $p+q\in P$. By (iii) and (iv), $a(p+q)=a\sp{+}-a\sp{-}=a$, 
whence $a\dg\leq p+q$ by Theorem \ref{th:Carrier} (iii). Let 
$r :=1-a\dg$. Then $r\in P$ and  $0=ar=a\sp{+}r-a\sp{-}r$, 
i.e., $a\sp{+}r=a\sp{-}r$. Consequently, $a\sp{+}r=pa\sp{+}r=
pa\sp{-}r=pqa\sp{-}r=0$ by (vii), and it follows that $pr=0$. 
Likewise, $qr=0$, so $(p+q)r=0$, and therefore $(p+q)a\dg=
(p+q)(1-r)=p+q$, i.e., $p+q\leq a\dg$ by Theorem \ref{th:e<=p}; 
hence $p+q=a\dg$.
\end{proof}

\begin{corollary} \label{co:monotoneSR}
If $0\leq a,b\in A$ and $aCb$, then $a\sp{2}\leq b\sp{2}
\Leftrightarrow a\leq b$.
\end{corollary}

\begin{proof}
Assume the hypotheses and suppose that $a\sp{2}\leq b\sp{2}$. As 
$0\leq(b-a)\sp{2}$, we have 
\setcounter{equation}{0} 
\begin{equation} \label{eq:a2b201}
0\leq(b-a)\sp{2}+b\sp{2}-a\sp{2}=2(b\sp{2}-ab),\text{\ whence\ } 
 0\leq(b-a)b.
\end{equation}
Also, by parts (vii), (viii), and (iii) of Theorem \ref{th:Carrier}, 
\begin{equation} \label{eq:a2b202} 
a\dg=(a\sp{2})\dg\leq(b\sp{2})\dg=b\dg,\text{\ whence\ \ }ab\dg=a. 
\end{equation}
Let $c :=(b-a)\sp{+}$ and $d :=(b-a)\sp{-}$. Then by Remarks 
\ref{rm:AbsProps} and parts (v) and (vi) of Theorem 
\ref{th:Prop a+a-}, $b\in C(b-a)\subseteq C(c)\cap C(d)$, and
we have
\begin{equation} \label{eq:a2b203}  
bCc,\ bCd,\ cCd,\ 0\leq c,\ 0\leq d,\ dc=0,\text{\ and\ }b-a=c-d.
\end{equation}
By (\ref{eq:a2b201}) and (\ref{eq:a2b203}),  
\begin{equation} \label{eq:a2b204}
0\leq(b-a)b=(c-d)b=cb-db.
\end{equation}
Since $dC(cb-db)$ and $0\leq d$, it follows from (\ref{eq:a2b204}),  
(\ref{eq:a2b203}), and [SA2] that $0\leq d(cb-db)=-d\sp{2}b$, i.e., 
$d\sp{2}b\leq 0$. Likewise, as $0\leq d\sp{2}$, $0\leq b$, and $bCd\sp{2}$, 
we also have $0\leq d\sp{2}b$; hence $d\sp{2}b=0$, and consequently 
\begin{equation} \label{eq:a2b205}
d\dg b=(d\sp{2})\dg b=0,\text{\ so\ }db=0,\text{\ whence\ }db\dg=0.
\end{equation}
As $c\in C(b)\subseteq C(b\dg)$, $0\leq c$, and $0\leq b\dg$, 
[SA2] implies that $0\leq cb\dg$, whence by (\ref{eq:a2b205}), 
(\ref{eq:a2b203}), and  (\ref{eq:a2b202}),  
\[
0\leq cb\dg=(c-d)b\dg=(b-a)b\dg=bb\dg-ab\dg=b-a. 
\]
Conversely, suppose that $a\leq b$, i.e., $0\leq b-a$. 
As $aCb$, we have $aC(b-a)$, and it follows from [SA2] 
that $0\leq a(b-a)=ab-a\sp{2}$, i.e., $a\sp{2}\leq ab$. 
Similarly, $0\leq (b-a)b=b\sp{2}-ab$, whence $ab\leq b\sp{2}$, 
and it follows that $a\sp{2}\leq b\sp{2}$. 
\end{proof}

\begin{definition}
If $a\in A$, then the \emph{signum} of $a$ is defined and denoted 
by $\sgn(a) :=(a\sp{+})\dg-(a\sp{-})\dg$. 
\end{definition}

\begin{theorem} \label{th:PolarDecomposition}
Let $a\in A$. Then: {\rm (i)} $\sgn(a)\in CC(a)$. {\rm (ii)} 
$\sgn(a)\sp{2}=a\dg$. {\rm (iii)} $\sgn(a)a=a\sgn(a)=|a|$. 
{\rm (iv)} $\sgn(a)|a|=|a|\sgn a=a$.
\end{theorem}

\begin{proof}
By Theorem \ref{th:Prop a+a-} (i), $C(a)\subseteq C((a\sp{+})\dg)
\cap C((a\sp{-})\dg)$, from which (i) follows. Part (ii) follows 
from parts (vii) and (viii) of Theorem \ref{th:Prop a+a-}, and 
parts (iii) and (iv) are consequences of parts (iii) and (iv) of 
Theorem \ref{th:Prop a+a-}. 
\end{proof}

The formula $a=\sgn(a)|a|=|a|\sgn(a)$ in Theorem 
\ref{th:PolarDecomposition} is called the \emph{polar decomposition} 
of $a$.

\begin{corollary} \label{co:StrongAxiii}
Let $a,b\in A$. Then:  {\rm (i)} $ab=0\Leftrightarrow |a||b|=0$. 
{\rm (ii)} $|a|\dg=a\dg$. 
\end{corollary}

\begin{proof}
We have $ab=0\Rightarrow|a||b|=\sgn(a)ab\sgn(b)=0$, and 
conversely, $|a||b|=0\Rightarrow ab=\sgn(a)|a||b|\sgn(b)=0$, 
proving (i). Arguing as above, we find that $|a|b=0\Leftrightarrow 
ab=0$, whence $|a|\dg=a\dg$, proving (ii).
\end{proof}

\section{Quadratic, Compression, and Sasaki Mappings} \label{sc:QCandSM} 

\begin{definition}
If $a\in A$, the mapping $J\sb{a}\colon A\to A$ defined by 
$J\sb{a}(b) :=aba$ for all $b\in A$ is called the \emph{quadratic 
mapping} determined by $a$. If $p\in P$, the quadratic mapping 
$J\sb{p}$ is called the \emph{compression} on $A$ with 
\emph{focus} $p$.
\end{definition}

\begin{theorem} \label{th:QuadLinOP}
If $a\in A$, then the quadratic mapping $J\sb{a}\colon \to A$ is 
both linear and order preserving.
\end{theorem}

\begin{proof}
Obviously, $J\sb{a}$ is linear.  Suppose that $0\leq h\in A$.
By [SA4], $0\leq|a|h|a|$, and we define $k :=(|a|h|a|)\sp{1/2}$. 
Thus, $k\sp{2}|a|\dg=|a|h|a||a|\dg=|a|h|a|=k\sp{2}$, so by (ii) 
and parts (vii) and (iii) of Theorem \ref{th:Carrier},
\setcounter{equation}{0} 
\begin{equation} \label{eq:aba01}
k\dg=(k\sp{2})\dg\leq|a|\dg=a\dg,\text{\ whence\ }ka\dg=k.
\end{equation} 
Let $w :=\sgn(a)$. Then by parts (ii) and (iv) of Theorem 
\ref{th:PolarDecomposition}, $w\sp{2}=a\dg$ and $a=w|a|=
|a|w$; hence by (\ref{eq:aba01})
\[
0\leq(wkw)\sp{2}=wkw\sp{2}kw=wka\dg kw=wk\sp{2}w=w|a|h|a|w
 =aha=J\sb{a}(h).
\]
Suppose $b,c\in A$ with $b\leq c$, and put $h :=c-b$. Then 
$0\leq h$, therefore $0\leq J\sb{a}(h)=J\sb{a}(c)-J\sb{a}(b)$, 
whence $J\sb{a}(b)\leq J\sb{a}(c)$, i.e., $J\sb{a}$ is order 
preserving.   
\end{proof}

Condition [SA4] requires that $a,b\in A\sp{+}\Rightarrow 
aba\in A\sp{+}$; however, by Theorem \ref{th:QuadLinOP}, we 
now have the stronger result $b\in A\sp{+}\Rightarrow aba\in A
\sp{+}$ for all $a\in A$.

\begin{lemma} \label{lm:Normaba}
Let $a,b\in A$ and $p\in P$. Then: {\rm (i)} $\|J\sb{a}(b)\|\leq
\|a\sp{2}\|\|b\|=\|a\|\sp{2}\|b\|$. {\rm (ii)} $J\sb{a}\colon 
A\to A$ is norm continuous. {\rm (iii)} If $p\not=0$, then $\|p\|=1$. 
{\rm (iv)} $\|J\sb{p}(a)\|\leq\|a\|$. 
\end{lemma}

\begin{proof}
As $-\|b\|\leq b\leq\|b\|$, we have 
\[
-\|b\|a\sp{2}=a(-\|b\|)a\leq aba\leq a\|b\|a=\|b\|a\sp{2},
\]
whence $\|aba\|\leq\|(\|b\|a\sp{2})\|=\|a\sp{2}\|\|b\|$. 
By Lemma \ref{lm:Norm} (ii), $\|a\sp{2}\|=\|a\|\sp{2}$, 
proving (i), and (ii) follows from (i). Also by Lemma \ref{lm:Norm} 
(ii), $\|p\|\sp{2}=\|p\sp{2}\|=\|p\|$, from which (iii) follows, 
and (iv) is a consequence of (i) and (iii).  
\end{proof}

Let $p\in P$ and $e\in E$. By Theorem \ref{th:QuadLinOP}, 
$J\sb{p}$ is linear and order preserving. Clearly, 
$J\sb{p}(1)=p\in P\subseteq E$. By Theorem \ref{th:e<=p}, 
$e\leq p\Rightarrow J\sb{p}(e)=e$. Also, if $J\sb{p}(e)=0$, then 
$pep=0$, whence $pe=ep=0$, so $e\leq 1-p.$ Conversely, as a 
consequence of \cite[Corollary 4.6]{FCPOAG}, compressions on 
$A$ are characterized as in the following theorem.

\begin{theorem} \label{th:Compression}
Let $J\colon A\to A$ be a linear and order-preserving mapping 
such that $J(1)\leq 1$ and, for every $e\in E$, $e\leq J(1)
\Rightarrow J(e)=e$. Then $p :=J(1)\in P$ and $J=J\sb{p}$.
\end{theorem}

\begin{lemma} \label{lm:CpNormClosed}
Let $a\in A$ and $p\in P$. Then: {\rm (i)} $a\in C(p)
\Leftrightarrow a=J\sb{p}(a)+J\sb{1-p}(a)$. {\rm (ii)} 
$C(p)$ is norm closed in $A$.
\end{lemma}

\begin{proof}
If $aCp$, it is clear that $a=pap+(1-p)a(1-p)$. Conversely, 
if $a=pap+(1-p)a(1-p)$ then $pa=pap=ap$, proving (i). Define 
the mapping $c\sb{p}\colon A\to A$ by $c\sb{p}(a) :=J\sb{p}(a)
+J\sb{1-p}(a)$. By Lemma \ref{lm:Normaba} (ii), $c\sb{p}$ is 
norm continuous, and by (i), $C(p)$ is its set of fixed points, 
proving (ii). 
\end{proof}

\begin{theorem} \label{th:AscendingLimit}
Let $(a\sb{n})\sb{n\in\Nat}$ be a sequence in $A$ and suppose 
that $\lim\sb{n\rightarrow\infty}a\sb{n}=a\in A$. Then: 
\begin{enumerate}
\item If $a\sb{n}\leq b\in A$ for all $n\in\Nat$, than $a
 \leq b$.
\item If $a\sb{1}\leq a\sb{2}\leq\cdots$, then $a$ is the 
 supremum {\rm (}least upper bound{\rm )} of $(a\sb{n})
 \sb{n\in\Nat}$ in $A$.
\item The positive cone $A\sp{+}$ is norm closed in $A$. 
\end{enumerate}
\end{theorem}

\begin{proof}
\setcounter{equation}{0}
By hypothesis, for each $m\in\Nat$, there exists $N\sb{m}\in\Nat$ 
such that, for all $n\in\Nat$, 
\begin{equation} \label{eq:Ascend01}
N\sb{m}\leq n\Rightarrow a\sb{n}-a\leq\|a\sb{n}-a\|\leq1/m
 \Rightarrow a\sb{n}\leq a+1/m.
\end{equation}
(i) Assume the hypothesis of (i). Then, for all $m\in\Nat$, 
$a-b\leq a-a\sb{N\sb{m}}$. Let $p :=((a-b)\sp{+})\dg\in CC(a-b)$. 
Then, $(a-b)\sp{+}=p(a-b)=p(a-b)p=J\sb{p}(a-b)$, so by Lemma 
\ref{lm:Normaba} (iv) and (\ref{eq:Ascend01}), for 
every $m\in\Nat$, 
\[
(a-b)\sp{+}=J\sb{p}(a-b)\leq J\sb{p}(a-a\sb{N\sb{m}})\leq
 \|J\sb{p}(a-a\sb{N\sb{m}})\|\leq\|a-a\sb{N\sb{m}}\|\leq 1/m,     
\]
whence $m(a-b)\sp{+}\leq 1$, and since $A$ is archimedean, 
it follows that $(a-b)\sp{+}\leq 0$. But $0\leq(a-b)\sp{+}$, 
so $(a-b)\sp{+}=0$, and consequently, $a-b=-(a-b)\sp{-}\leq 
0$, i.e., $a\leq b$.

\smallskip

(ii) By (\ref{eq:Ascend01}), for each $m\in\Nat$,  
\[
a\sb{1}\leq a\sb{2}\leq\cdots\leq a\sb{N\sb{m}}\leq a+1/m;
\] 
hence $a\sb{n}\leq a+1/m$ for all $n\in\Nat$. Therefore, for each 
$n\in\Nat$, we have $m(a\sb{n}-a)\leq 1$ for all $m\in\Nat$, and 
since $A$ is archimedean, it follows that $a\sb{n}-a\leq 0$, i.e., 
$a\sb{n}\leq a$. If $a\sb{n}\leq b\in A$ for all $n\in\Nat$, then 
by (i), $a\leq b$; hence $a$ is the least upper bound of 
$(a\sb{n})\sb{n\in\Nat}$.

\smallskip(iii) Let $(c\sb{n})\sb{n\in\Nat}$ be a sequence in 
$A\sp{+}$ and suppose that $c\sb{n}\rightarrow c$. Then 
$-c\sb{n}\rightarrow -c$, and as $-c\sb{n}\leq 0$ for all $n\in\Nat$, 
(i) implies that $-c\leq 0$, i.e., $c\in A\sp{+}$.
\end{proof}

By combining the quadratic mapping $J\sb{a}$ with the carrier, we 
obtain the \emph{Sasaki mapping} on $A$ as per the following definition.

\begin{definition} \label{df:SasakiMap}
For each $a\in A$, the \emph{Sasaki mapping}\footnote{The 
terminology derives from the fact that, for $p\in P$, the 
restriction of $\phi\sb{p}$ to $P$ is a so-called \emph{Sasaki 
projection} on $P$ \cite[p. 99]{Kalm}. See Theorem \ref{th:PisanOML} 
below.} $\phi\sb{a}\colon A\to P$ is defined by $\phi\sb{a}(b) :=
(J\sb{a}(b))\dg=(aba)\dg$ for all $b\in B$. 
\end{definition}  

\begin{theorem} \label{th:phiProps}
Let $a,b,c\in A$. Then: {\rm (i)} $\phi\sb{a}(b)\leq\phi\sb{a}
(1)=a\dg$. {\rm (ii)} $0\leq b\leq c\Rightarrow\phi\sb{a}(b)
\leq\phi\sb{a}(c)$. {\rm (iii)} If $0\leq b$, then $\phi\sb{a}
(b)c=0\Rightarrow \phi\sb{a}(c)b=0$. {\rm (iv)} If $0\leq b,c$, 
then $\phi\sb{a}(b)c=0\Leftrightarrow\phi\sb{a}(c)b=0$. {\rm (v)} 
If $0\leq b$, then $\phi\sb{a}(b)=\phi\sb{a}(b\dg)$. 
\end{theorem}

\begin{proof}
(i) As $abaa\dg=aba$, Theorem \ref{th:Carrier} (iii) implies 
that $\phi\sb{a}(b)=(aba)\dg\leq a\dg$. Also, $\phi\sb{a}(1)=
(a\sp{2})\dg=a\dg$ by Theorem \ref{th:Carrier} (vii).

\smallskip

(ii) Assume that $0\leq b\leq c$. Then $0\leq J\sb{a}(b)\leq 
J\sb{a}(c)$, so $\phi\sb{a}(b)\leq\phi\sb{a}(c)$ by  
Theorem \ref{th:Carrier} (viii).

\smallskip

(iii) Suppose that $0\leq b$ and $\phi\sb{a}(b)c=0$. 
Then $(aba)\dg c=0$, whence $abac=0$, and therefore 
$(aca)b(aca)=ac(abac)a=0$, whereupon $acab=0$ by [SA5], 
and it follows that $(aca)\dg b=0$, i.e., $\phi\sb{a}(c)b=0$.

\smallskip

(iv) Follows from (iii).

(v) Suppose that $0\leq c$. We have $\phi\sb{a}(c)b
=0\Leftrightarrow\phi\sb{a}(c)b\dg=0$, and as $0\leq b\dg$, 
it follows from (iv) that $\phi\sb{a}(c)b\dg=0\Leftrightarrow
\phi\sb{a}(b\dg)c=0$.  Consequently, $\phi\sb{a}(b)c=0
\Leftrightarrow\phi\sb{a}(b\dg)c=0$. 
Putting $c=1-\phi\sb{a}(b\dg)$, we find that $\phi\sb{a}(b)=
\phi\sb{a}(b)\phi\sb{a}(b\dg)$, hence $\phi\sb{a}(b)\leq 
\phi\sb{a}(b\dg)$. Similarly, putting $c=1-\phi\sb{a}(b)$, 
we obtain $\phi\sb{a}(b\dg)\leq\phi\sb{a}(b)$. 
\end{proof}

\begin{theorem} \label{th:vAv}
Let $0\not=v\in P$ and define $vAv :=J\sb{v}(A)=\{vav:a\in A\}
=\{b\in A:b=bv=vb\}$. Then $vAv$ is norm-closed in $A$ and, with 
the partial order inherited from $A$, $vAv$ is a synaptic 
algebra with unit $v$ and enveloping algebra $vRv$.\footnote{In 
dealing with the synaptic algebra $vAv$ in the presence of the 
synaptic algebra $A$, we cannot follow the convention (previously 
adopted for $A$) of identifying real numbers $\lambda$ with 
multiples $\lambda v$ of the unit element $v$.} Moreover, the 
order-unit norm on $vAv$ is the restriction to $vAv$ of the order-unit 
norm on $A$, and for all $a,b\in vAv$, we have: $a\circ b,\,a\dg,\,|a|,
\,a\sp{+},\,a\sp{-},\,\sgn(a)\in vAv$; $J\sb{a}(A)\subseteq vAv$; 
$\phi\sb{a}(A)\subseteq vAv$; and $0\leq a\Rightarrow a\sp{1/2}\in vAv$.
\end{theorem}

\begin{proof}
By Lemma \ref{lm:Normaba} (ii), $J\sb{v}\colon A\to A$ is norm 
continuous, and since $vAv$ is the set of fixed points of $J\sb{v}$, 
it follows that $vAv$ is a norm-closed linear subspace of $A$. Let  
$b\in vAv$. Then there exists $n\in\Nat$ such that $b\leq n=n1$; hence  
$b=J\sb{v}(b)\leq nJ\sb{v}(1)=nv$, so $v$ is an order unit in $vAv$. 
By a similar argument, if $0\leq\lambda\in\reals$, then $-\lambda
\leq b\leq\lambda\Rightarrow -\lambda v\leq b\leq\lambda v$; 
conversely, $-\lambda v\leq b\leq\lambda v\Rightarrow-\lambda
\leq b\leq\lambda$ follows from the fact that $0\leq v\leq 1$; 
hence $\|b\|=\inf\{0<\lambda\in\reals:-\lambda v\leq b\leq
\lambda v\}$. Thus, [SA1] holds for $vAv$. 

That $vAv$ satisfies [SA2]--[SA5] is obvious. If $0\leq b\in vAv$, 
then, since $b=bv=vb$ and $b\sp{1/2}\in CC(b)$, we have $0\leq 
vb\sp{1/2}v=vb\sp{1/2}=b\sp{1/2}v$ with $(vb\sp{1/2})\sp{2}=b$; 
hence $b\sp{1/2}=vb\sp{1/2}$ by the uniqueness of square roots 
(Theorem \ref{th:UniqueSR}), and it follows that $b\sp{1/2}\in 
vAv$. Thus, $vAv$ satisfies [SA6]. If $b\in vAv$, we again have 
$b=bv=vb$, whence $b\dg\leq v$, and since $b\dg\in CC(b)$, 
it follows easily that $b\dg\in vAv$. Thus, $vAv$ satisfies [SA7].

To show that $vAv$ satisfies [SA8], suppose that $v\leq b\in vAv$. 
Then $1=v+(1-v)\leq b+1-v$ with $b=vb=bv$. By [SA8], there exists 
$c\in A$ such that $1=c(b+1-v)=(b+1-v)c$.  Applying $J\sb{v}$ to 
both sides of the latter equation, we find that $v=vcbv=vbcv=
vcvb=bvcv$, and since $vcv\in vAv$, it follows that $vAv$ satisfies 
[SA8].  Obviously, $vAv$ inherits condition $[SA9]$ from $A$. We 
omit the completely straightforward proofs of the remaining 
assertions of the theorem.      
\end{proof}

\section{Orthomodularity of the Projection Lattice} 
\label{sc:OML} 

\begin{definition} \label{df:perp}
The mapping $\sp{\perp}\colon P\to P$ is defined by $p\sp{\perp}
:=1-p$ for all $p\in P$. If $p,q\in P$, we say that $p$ is 
\emph{orthogonal} to $q$, in symbols $p\perp q$, iff $p\leq 
q\sp{\perp}$.
\end{definition}

We note that $p\perp q\Rightarrow q\perp p$ and that $p\perp p
\Leftrightarrow p=0$. In this section we are going to prove that, 
with $p\mapsto p\sp{\perp} :=1-p$ as the \emph{orthocomplementation}, 
$P$ is a \emph{orthomodular lattice} as per the following 
definition \cite{Beran, Kalm}.

\begin{definition} \label{df:OMP,OML}
Let $X$ be a partially ordered set (poset). A mapping 
$x\mapsto x\sp{\perp}$ from $X$ to $X$ is called an 
\emph{involution} iff it is order reversing ($x\leq y
\Rightarrow y\sp{\perp}\leq x\sp{\perp}$) and of period 
$2$ ($(x\sp{\perp})\sp{\perp}=x$) for all $x,y\in X$. 
An \emph{orthomodular poset} (OMP) is a partially ordered 
set $X$ with a smallest element $0$, a largest element $1$, 
and an involution $\sp{\perp}\colon X\to X$, called the 
\emph{orthocomplementation}, such that, for all $x,y\in X$: 
(i) The infimum (greatest lower bound) $x\wedge x\sp{\perp}$ 
of $x$ and $x\sp{\perp}$ exists in $X$ and $x\wedge x
\sp{\perp}=0$. (ii) If $x\leq y\sp{\perp}$, then the supremum 
(least upper bound) $x\vee y$ exists in $X$. (iii) If $x\leq y$, 
then $y=x\vee(x\sp{\perp}\wedge y)$. An \emph{orthomodular lattice} 
(OML) is an OMP $X$ that is a lattice (i.e., every pair $x,y\in X$ 
has an infimum $x\wedge y$ and a supremum $x\vee y$ in $X$.) 
\end{definition}

Let $X$ be a poset and let $a,b,x,y\in X$. If we write 
$a=x\wedge y$, or $x\wedge y=a$, we mean that the infimum 
(greatest lower bound) $x\wedge y$ of $x$ and $y$ exists in 
$X$ and that it equals $a$. A similar convention applies to an  
existing supremum (least upper bound) $b=x\vee y$ of $x$ and $y$ 
in $X$.  An involution $x\mapsto x\sp{\perp}$ on $X$ gives rise 
to a \emph{De\,Morgan duality} on $X$ whereby existing infima are 
converted to suprema and \emph{vice versa}. For instance, if 
$a=x\wedge y$, then $a\sp{\perp}=x\sp{\perp}\vee y\sp{\perp}$. 
Also, if $X$ has a smallest element $0$ and a largest element 
$1$, then $0\sp{\perp}=1$ and $1\sp{\perp}=0$. Obviously, the 
mapping $p\mapsto p\sp{\perp} =1-p$ (respectively, $e\mapsto 1-e$) 
is an involution on the poset $P$ (respectively, on the poset $E$),  
and $a\mapsto -a$ is an involution on $A$. 

Suppose that $X$ is an OMP with $x\mapsto x\sp{\perp}$ as the 
orthocomplementation. Then by Definition \ref{df:OMP,OML} (i) 
and De\,Morgan duality, we have both $x\wedge x\sp{\perp}=0$ 
and $x\vee x\sp{\perp}=1$, i.e., $x\sp{\perp}$ is an \emph
{orthogonal complement}, or for short, an \emph{orthocomplement} 
of $x$ in $X$. Let $x,y\in X$ with $x\leq y$. Then $x\leq(y\sp
{\perp})\sp{\perp}$, whence $x\vee y\sp{\perp}$ exists in $X$ 
by Definition \ref{df:OMP,OML} (ii), and therefore $x\sp{\perp}
\wedge y=(x\vee y\sp{\perp})\sp{\perp}$ exists in $X$ by 
De\,Morgan duality. Since $x\leq x\vee y\sp{\perp}=
(x\sp{\perp}\wedge y)\sp{\perp}$, it also follows from 
Definition \ref{df:OMP,OML} (ii) that the supremum $x\vee
(x\sp{\perp}\wedge y)$ exists in $X$. The condition $x\leq y
\Rightarrow y=x\vee(x\sp{\perp}\wedge y)$ in Definition 
\ref{df:OMP,OML} (iii) is called the \emph{orthomodular 
law}.  

\begin{lemma} \label{lm:PisanOMP}
For all $p,q\in P$: {\rm (i)} $pCq\Rightarrow pq=p\wedge q$. 
{\rm (ii)} $p\perp q\Leftrightarrow pq=0$. {\rm (iii)} $p
\perp q\Rightarrow p\vee q=p+q$. {\rm (iv)} $p\leq q
\Rightarrow q-p=p\sp{\perp}\wedge q\in P$. {\rm (v)} 
With $p\mapsto p\sp{\perp} :=1-p$ as the orthocomplementation, 
$P$ is an OMP. 
\end{lemma}

\begin{proof}
(i) Assume that $pq=qp$. Obviously, $(pq)\sp{2}=pq$, so $pq\in P$. 
Also $p(pq)=pq$ and $q(pq)=pq$, so $pq\leq p,\,q$ by Theorem 
\ref{th:e<=p}. Suppose that $r\in P$ and $r\leq p,\,q$. Again by 
Theorem \ref{th:e<=p}, $rp=pr=r$ and $rq=qr=r$, whence $rpq=r$, 
i.e., $r\leq pq$. Therefore $pq=p\wedge q$.

\smallskip

(ii) By Theorem \ref{th:e<=p}, $p\perp q\Leftrightarrow 
p\leq1-q\Leftrightarrow p=p(1-q)=p-pq\Leftrightarrow pq=0.$

\smallskip

(iii) Suppose that $p\perp q$ Then $pq=0$ by (ii), 
so $qp=0$ by Lemma \ref{lm:ab=0} (iv), and it follows that 
$(p+q)\sp{2}=p\sp{2}+q\sp{2}=p+q$, i.e., $p+q\in P$. As $0
\leq p,q$, we have $p,q\leq p+q$. Suppose that $r\in P$ with 
$p,q\leq r$. Then, by Theorem \ref{th:e<=p}, $p=pr$ and $q=qr$, 
whereupon $p+q=(p+q)r$, i.e., $p+q\leq r$. Therefore, 
$p+q=p\vee q$.

\smallskip 

(iv) Suppose that $p\leq q=(q\sp{\perp})\sp{\perp}$. 
Then by (iii), $p+q\sp{\perp}=p\vee q\sp{\perp}\in P$, whence 
$(p+q\sp{\perp})\sp{\perp}=p\sp{\perp}\wedge q\in P$. But 
$(p+q\sp{\perp})\sp{\perp}=1-(p+1-q)=q-p$.

\smallskip

(v) Obviously, $0$ is the smallest element and $1$ 
is the largest element in the poset $P$. In view of (ii), it 
remains only to show that the orthomodular law holds in $P$. 
But, if $p,q\in P$ with $p\leq q$, then by (iv), $q-p=p
\sp{\perp}\wedge q$ and by (iii), $q=p+(q-p)=p+(p\sp{\perp}
\wedge q)=p\vee(p\sp{\perp}\wedge q)$.
\end{proof}

\begin{theorem} \label{th:PreserveSup}
Let $a\in A$. Then: {\rm (i)} If $p,q\in P$, then 
$\phi\sb{a}(p)\perp q\Leftrightarrow p\perp\phi\sb{a}(q)$.
{\rm (ii)} $\phi\sb{a}$ preserves all existing suprema in 
$P$, i.e., if $Q\subseteq P$ and $r=\bigvee Q$, then 
$\phi\sb{a}(r)=\bigvee\{\phi\sb{a}(q):q\in Q\}$. 
\end{theorem}

\begin{proof}
Part (i) follows from Theorem \ref{th:phiProps} (iv) and Lemma 
\ref{lm:PisanOMP} (ii). To prove part (ii), suppose that $Q
\subseteq P$ and $r=\bigvee Q$. Then, for all $q\in Q$, $0\leq r
\leq q$, whence $\phi\sb{a}(q)\leq\phi\sb{a}(r)$ by Theorem 
\ref{th:phiProps} (ii). Suppose that $t\in P$ and $\phi\sb{a}(q)
\leq t$ for all $q\in Q$. Then, for all $q\in Q$, $\phi\sb{a}(q)
\perp t\sp{\perp}$, whence, by (i), $q\perp\phi\sb{a}(t\sp{\perp})$, 
i.e., $q\leq(\phi\sb{a}(t\sp{\perp}))\sp{\perp}$, and it follows 
that $r\leq(\phi\sb{a}(t\sp{\perp}))\sp{\perp}$. Consequently, by 
(i) again, $\phi\sb{a}(r)\perp t\sp{\perp}$, i.e., $\phi\sb{a}(r)
\leq t$, and therefore $\phi\sb{a}(r)=\bigvee\{\phi\sb{a}(q):q\in Q\}$. 
\end{proof}

\begin{lemma} \label{lm:phisbp}
Let $p,q,r\in P$. Then: {\rm (i)} $\phi\sb{p}(r)\leq p$. 
{\rm (ii)} $r\leq p\Leftrightarrow\phi\sb{p}(r)=r$. 
{\rm (iii)} $r\perp p\Leftrightarrow\phi\sb{p}(r)=0$.
{\rm (iv)} $p\wedge q$ exists in $P$ and $p\wedge q=
p-\phi\sb{p}(q\sp{\perp})$. 
\end{lemma}

\begin{proof}
\setcounter{equation}{0}
(i) By Theorems \ref{th:phiProps} (i) and \ref{th:Carrier} (ii), 
$\phi\sb{p}(r)\leq p\dg=p$.

\smallskip

(ii) If $r\leq p$, then $r=pr=rp$, so $\phi\sb{p}(r)
=(prp)\dg=r\dg=r$. The converse implication follows from (i). 

\smallskip

(iii) By Lemma \ref{lm:PisanOMP} (ii), [SA5], 
and Theorem \ref{th:Carrier} (i), $p\perp r\Leftrightarrow 
pr=0\Leftrightarrow prp=0\Leftrightarrow(prp)\dg=0\Leftrightarrow 
\phi\sb{p}(r)=0$.

\smallskip

(iv) Let $t :=(\phi\sb{p}(q\sp{\perp}))\sp{\perp}=1
-\phi\sb{p}(q\sp{\perp})\in P$. By (i), $\phi\sb{p}(q\sp{\perp})
\leq p$, whence $pC\phi\sb{p}(q\sp{\perp})$ and $p\phi\sb{p}
(q\sp{\perp})=\phi\sb{p}(q\sp{\perp})$. Therefore, by parts (i) and 
(iv) of Lemma \ref{lm:PisanOMP},
\begin{equation} \label{eq:OML01}
p\wedge t=pt=tp=p-\phi\sb{p}(q\sp{\perp})=p\wedge(\phi\sb{p}(q\sp{\perp}))
\sp{\perp}\in P.
\end{equation}
By (\ref{eq:OML01}), $p\wedge t\perp\phi\sb{p}(q\sp{\perp})$, whence 
by Theorem \ref{th:PreserveSup} (i), $\phi\sb{p}(p\wedge t)\perp 
q\sp{\perp}$, i.e. $\phi\sb{p}(p\wedge t)\leq q$. But, $p\wedge t
\leq p$, whence by (ii), $\phi\sb{p}(p\wedge t)=p\wedge t$, and we 
have $p\wedge t\leq q$. Thus $p\wedge t\leq p,q$. Suppose $r\in P$ 
and $r\leq p,q$. By (ii), $\phi\sb{p}(r)=r\leq q$, so $\phi\sb{p}(r)
\perp q\sp{\perp}$, and therefore $r\perp\phi\sb{p}(q\sp{\perp})$ by 
Theorem \ref{th:PreserveSup} (i); hence, $r\leq(\phi\sb{p}(q
\sp{\perp}))\sp{\perp}=t$. But $r\leq p$; hence $r\leq p\wedge t$ 
by (\ref{eq:OML01}), and it follows that $p\wedge t=p\wedge q$.
\end{proof}

\begin{theorem} \label{th:PisanOML}
$P$ is an OML and, for all $p,q\in P$, $\phi\sb{p}(q)=
p\wedge(p\sp{\perp}\vee q)$.
\end{theorem}

\begin{proof}
Let $p,q\in P$. Then by Lemma \ref{lm:phisbp} (iv), $p\wedge q$ 
exists in $P$, so by De\,Morgan duality, $p\vee q=(p\sp{\perp}
\wedge q\sp{\perp})\sp{\perp}$ also exists in $P$. Therefore, 
$P$ is an OML. Also, as $p\sp{\perp}\leq p\sp{\perp}\vee q$, 
we have $p\sp{\perp}\vee q=p\sp{\perp}\vee(p\wedge(p\sp{\perp}
\vee q))$ by the orthomodular law; hence, by Theorem 
\ref{th:PreserveSup} (ii) and parts (iii) and (ii) of Lemma 
\ref{lm:phisbp},
\[ 
\phi\sb{p}(q)=\phi\sb{p}(p\sp{\perp})\vee\phi\sb{p}(q)=
 \phi\sb{p}(p\sp{\perp}\vee q)=\phi\sb{p}(p\sp{\perp}
 \vee(p\wedge(p\sp{\perp}\vee q)))
\]
\[ 
 =\phi\sb{p}(p\sp{\perp})\vee\phi\sb{p}(p\wedge(p\sp{\perp}\vee q))
 =p\wedge(p\sp{\perp}\vee q).\hspace{.5 in} \qedhere
\]
\end{proof}

Two elements $p$ and $q$ of an orthomodular lattice are said to 
be \emph{compatible} (or to \emph{commute}) iff $p=(p\wedge q)\vee 
(p\wedge q\sp{\perp})$ \cite[p. 20]{Kalm}. By a standard argument 
(e.g., \cite[Theorem 3.11]{FPMonotone}), if $p,q\in P$, then 
$p$ and $q$ are compatible in the foregoing sense iff $pCq$.

\section{Synaptic Versus GH-Algebras} \label{sc:SGH} 

Every generalized Hermitian (GH) algebra $G$ \cite[Definition 2.1]
{GHAlg1} is a synaptic algebra. Indeed, [SA1] follows from 
\cite[Theorem 4.2]{GHAlg1} and parts (ii), (iii), and (iv) of 
\cite[Definition 2.1]{GHAlg1} imply [SA2]--[SA5]. Also, [SA6] 
follows from \cite[Theorem 4.5]{GHAlg1}, \cite[Theorem 5.2]{GHAlg1} 
implies [SA7], and [SA8] is a consequence of \cite[Lemma 4.1]{GHAlg2}.
Finally, by \cite[Lemma 6.6 (iii)]{GHAlg1}, $G$ satisfies [SA9]; 
hence $G$ is a synaptic algebra.

By \cite[Definition 2.1 (vii)]{GHAlg1}, a generalized Hermitian 
algebra $G$ has the following \emph{commutative Vigier\,\footnote{See 
\cite[Section 5]{FSRI} for the origin of the terminology} property}:
 
\medskip

\noindent{\bf [CV]} \emph{Every bounded ascending sequence $g\sb{1}
\leq g\sb{2}\leq\cdots$ of pairwise commuting elements in $G$ 
has a supremum $g$ in $G$ and $g\in CC(\{g\sb{n}:n\in\Nat\})$}.

\medskip

\noindent Clearly, a synaptic algebra $A$ is a GH-algebra iff it 
satisfies [CV]. The condition [CV] is quite strong\,\footnote{For 
instance, as a consequence of [CV], the orthomodular lattice of 
projections in a GH-algebra is necessarily $\sigma$-complete 
\cite[Theorem 5.4]{GHAlg1}.} (see \cite[Section 4]{GHAlg1}), and the 
main impetus for the formulation in Definition \ref{df:SynapticAlgebra} 
is to replace [CV] by some of its algebraic consequences [SA6], [SA7], 
and [SA8], accompanied by the considerably weaker condition [SA9]. 

As an indication of the extent to which synaptic algebras generalize 
GH-algebras, we may consider the commutative case. The projections 
in a commutative GH-algebra form a $\sigma$-complete Boolean 
algebra; moreover, every $\sigma$-complete Boolean algebra can be   
realized as the (Boolean) lattice of projections in a commutative 
GH-algebra \cite[Theorem 5.7]{GHAlg2}. On the other hand, the 
projections in a commutative synaptic algebra form a Boolean algebra, 
which need not be $\sigma$-complete; moreover, \emph{every Boolean 
algebra $B$ can be realized as the {\rm (}Boolean{\rm )} lattice of 
projections in a commutative synaptic algebra.} Indeed, by Stone's 
theorem, $B$ can be represented as the field ${\mathcal F}$ of compact 
open subsets of a totally disconnected Hausdorff space $X$, and the 
projection lattice of the commutative synaptic algebra $A$ in 
Example \ref{ex:functions} is isomorphic to $B$.

\section{Invertible and Regular Elements} \label{sc:IRE} 

As we now show, the results in \cite[Section 4]{GHAlg2} pertaining 
to invertible and von Neumann regular elements of a GH-algebra $G$ 
go through for our synaptic algebra $A$, although we must be a 
little careful since the proof of \cite[Lemma 4.1]{GHAlg2} depends 
on the property [CV]. As usual, an element $a\in A$ is 
\emph{invertible} iff there exists a (necessarily unique) element 
$a\sp{-1}\in A$ such that $aa\sp{-1}=a\sp{-1}a=1$. If $a$ is 
invertible, it is clear that $a\sp{-1}\in CC(a)$ and that $a\dg=1$.   

\begin{lemma} \label{lm:AbsInv}
Let $a\in A$. Then: {\rm (i)} If $0\leq a$ and $a$ is 
invertible, then $0\leq a\sp{-1}$.  {\rm (ii)} $a$ is 
invertible iff $|a|$ is invertible, and if $a$ is invertible, 
then $|a|\sp{-1}=|a\sp{-1}|$. 
\end{lemma} 

\begin{proof}
(i) Suppose $0\leq a$ and $a$ is invertible. As $aC(a\sp{-1})
\sp{2}$ and $0\leq(a\sp{-1})\sp{2}$, [SA2] implies that 
$0\leq a(a\sp{-1})\sp{2}=a\sp{-1}$.

(ii) Let $s :=\sgn(a)$. By Theorem \ref{th:PolarDecomposition}, 
$s\in CC(a)$, $s\sp{2}=a\dg$, $sa=as=|a|$, and $s|a|=|a|s=a$. 
Suppose $a$ is invertible. As $s\in CC(a)$, we have $sCa\sp{-1}$ 
and $|a|(sa\sp{-1})=(sa\sp{-1})|a|=1$; hence $|a|$ is invertible 
and $|a|\sp{-1}=sa\sp{-1}$. Also, $s\sp{2}=a\dg=1$, and by (i), 
$0\leq sa\sp{-1}$. But, $(sa\sp{-1})\sp{2}=s\sp{2}(a\sp{-1})\sp{2}=
(a\sp{-1})\sp{2}$, whence $|a|\sp{-1}=sa\sp{-1}=|sa\sp{-1}|=
|a\sp{-1}|$. Conversely, if $|a|$ is invertible, it is clear that 
$a$ is invertible with $a\sp{-1}=s|a|\sp{-1}$.
\end{proof}

\begin{theorem} \label{th:Invertibility}
If $a\in A$, then $a$ is invertible iff there exists $0<\epsilon
\in\reals$ such that $\epsilon\leq|a|$.
\end{theorem}

\begin{proof}
Suppose first that $a$ is invertible. Then, by Lemma \ref{lm:AbsInv}, 
$|a|$ is invertible. As $1$ is an order unit, there exists 
$n\in\Nat$ such that $|a|\sp{-1}\leq n$, and since $|a|$ commutes 
with $n-|a|\sp{-1}$, [SA2] implies that $0\leq(n-|a|\sp{-1})|a|$, 
i.e., $1\leq n|a|$. Consequently, with $0<\epsilon :=1/n$, we have 
$\epsilon\leq|a|$.

Conversely, suppose $0<\epsilon\leq|a|$. Then $1\leq\epsilon
\sp{-1}|a|$; hence by [SA8], $\epsilon\sp{-1}|a|$ is invertible, 
and it follows that $|a|$ is invertible with $|a|\sp{-1}=
\epsilon\sp{-1}(\epsilon\sp{-1}|a|)\sp{-1}$. Thus $a$ is 
invertible by Lemma \ref{lm:AbsInv}.
\end{proof}

\begin{definition} \label{df:Regular}
Let $a\in A$. (i) $a$ is \emph{von Neumann regular} iff there 
exists $b\in A$ such that $ab,ba\in A$ and $aba=a$. (ii) $a$ is 
\emph{regular} iff there exists $0<\epsilon\in\reals$ such that 
$\epsilon a\dg\leq|a|$.
\end{definition}

Obviously, $0$ is both von Neumann regular and regular. The proof 
of the following theorem is virtually identical\,\footnote{Note 
that $a\dg Aa\dg$ is a synaptic algebra by Theorem \ref{th:vAv}.} 
to the proof of \cite[Theorem 4.5]{GHAlg2}.

\begin{theorem} \label{th:Regularity}
If $0\not=a\in A$, then the following conditions are mutually equivalent: 
{\rm (i)} $a$ is von Neumann regular. {\rm (ii)} There exists $r\in a
\dg Aa\dg$ such that $ar=ra=a\dg$. {\rm (iii)} $a$ is invertible in 
the synaptic algebra $a\dg Aa\dg$. 
{\rm (iv)} $a$ is regular.
\end{theorem}

\begin{corollary} \label{co:invert}
If $a\in A$, then $a$ is invertible iff $a$ is regular and $a\dg=1$.
\end{corollary}

If $0\not=a\in A$ and $a$ is regular, then the (necessarily unique) 
inverse of $a$ in $a\dg Aa\dg$ (Theorem \ref{th:Regularity}) is 
called the \emph{pseudo-inverse} of $a$ in $A$, and by definition, 
the pseudo-inverse of $0$ is $0$. If $a$ is regular, it is not 
difficult to show that the pseudo-inverse of $a$ belongs to $CC(a)$.

\begin{theorem} \label{th:RegPosNeg}
If  $a\in A$, then $a$ is regular iff both $a\sp{+}$ and 
$a\sp{-}$ are regular.
\end{theorem}

\begin{proof}
Let $p :=(a\sp{+})\dg$ and $q :=(a\sp{-})\dg$. Then by Theorem 
\ref{th:Prop a+a-}, $p,q\in CC(a)$, $p+q=a\dg$, $pq=pa\sp{-}=0$, 
$qp=qa\sp{+}=0$, $pa=pa\sp{+}=a\sp{+}$, and $qa=qa\sp{-}=
a\sp{-}$.

Suppose that $a$ is regular. Then there exists $0<\epsilon\in
\reals$ with $\epsilon(p+q)=\epsilon a\dg\leq|a|=a\sp{+}+
a\sp{-}$, so $\epsilon p=p(\epsilon(p+q))\leq p(a\sp{+}+a
\sp{-})=a\sp{+}=|a\sp{+}|$,  whence $a\sp{+}$ is regular. 
Likewise, $\epsilon q\leq a\sp{-}$, so $a\sp{-}$ is regular. 
Conversely, if both $a\sp{+}$ and $a\sp{-}$ are regular, there 
exist $0<\alpha,\beta$ such that $\alpha p\leq a\sp{+}$ and 
$\beta q\leq a\sp{-}$; hence with $\epsilon :=\min\{\alpha,\beta\}$, 
we have $\epsilon a\dg=\epsilon(p+q)\leq a\sp{+}+a\sp{-}=|a|$, and 
it follows that $a$ is regular.
\end{proof}

\begin{corollary} \label{co:Invertible}
$a\in A$ is invertible iff $a\dg=1$ and both $a\sp{+}$ and 
$a\sp{-}$ are regular.
\end{corollary}

\section{Spectral Resolution} \label{sc:Spectral} 

In this section, we show that the synaptic algebra $A$ is a 
so-called \emph{spectral order-unit normed space}; hence the 
results of \cite{SOUS} are at our disposal. In particular, 
every element in $A$ both determines and is determined by a 
family of projections---its \emph{spectral resolution}.

As per \cite[Definition 1.5 (i)]{SOUS}, an element $a\in A$ 
is \emph{compatible} with a projection $p\in P$ iff $a=
J\sb{p}(a)+J\sb{1-p}(a)$. Thus, by Lemma \ref{lm:CpNormClosed} 
(i), $C(p)$ is the set of all elements of $A$ that are 
compatible with $p$; hence, the notation used in 
\cite[Definition 1.5 (i) and ff.]{SOUS} is consistent with our 
notation in this article.

\begin{theorem} \label{th:CompressionBase}
The family $\cb$ is a spectral compression base {\rm\cite[Definition 1.7]
{SOUS}} for the order-unit space $A$.
\end{theorem}

\begin{proof}
To begin with, we have to show that $P$ is a normal sub-effect algebra 
of $E$ (\cite[Definition 1]{FCBinUG}). Of course, $0,1\in P$, and 
$p\in P\Rightarrow 1-p=p\sp{\perp}\in P$. Also, if $p,q\in P$ with 
$p+q\leq 1$, then $p+q=p\vee q\in P$ by Lemma \ref{lm:PisanOMP} (iii). 
Therefore $P$ is a sub-effect algebra of $E$. Suppose that 
$d,e,f,d+e+f\in E$ with $p :=d+e\in P$ and $q :=d+f\in P$. Then 
$e+q=d+e+f\leq 1$, so $e\leq 1-q$, and therefore by Theorem 
\ref{th:e<=p}, $e=e(1-q)$, i.e., $eq=0$. Also, $d\leq d+f=q$, 
so $dq=d$ by Theorem \ref{th:e<=p}, and it follows that $pq=(d+e)q
=dq=d$. By symmetry, $qp=d$; hence by Lemma \ref{lm:PisanOMP} (i), 
$d=p\wedge q\in P$, and it follows that $P$ is a normal sub-effect 
algebra of $E$.

Now let $p,q,r\in P$ with $p+q+r\leq 1$. Then $pq=pr=qr=0$, $p+r=
p\vee r\in P$ and $q+r=q\vee r\in P$, whence, for all $a\in A$,
\[
J\sb{p+r}(J\sb{q+r}(a))=(p+r)(q+r)a(q+r)(p+r)=rar=J\sb{r}(a),
\]
and it follows that $\cb$ is a compression base for $A$ 
(\cite[Definition 2]{FCBinUG}).

If $e\in E$, then by Theorem \ref{th:Carrier} (vii), $e\dg$ is 
the smallest projection $p$ such that $e\leq p$\,; hence the 
compression base $\cb$ has the projection cover property 
(\cite[Definition 1.4]{SOUS}).

Let $a\in A$ and let $p :=(a\sp{+})\dg$. Then by parts (i), 
(iii), and (vi) of Theorem \ref{th:Prop a+a-}, we have 
$C(a)\subseteq C(p)$, $J\sb{p}(a)=pap=pa=a\sp{+}\geq 0$, 
and $J\sb{1-p}(a)=(1-p)a(1-p)=(1-p)a=a-pa=a-a\sp{+}=
-a\sp{-}\leq 0$. Thus, the compression base $\cb$ 
has the comparability property (\cite[Definition 1.6]{SOUS}),
and therefore $\cb$ is a spectral compression base for $A$.
\end{proof}

If $a\in A$, it is clear that, for all $p\in P$, 
\[
p\leq 1-a\dg\Leftrightarrow a\dg p=0\Leftrightarrow ap=pa=0
 \Leftrightarrow a\in C(p)\text{\ with\ }J\sb{p}(a)=pap=0.
\]
Therefore, as per \cite[Theorem 2.1 and ff.]{SOUS}, the mapping 
$'\colon A\to P$ defined by $a\,' :=1-a\dg$ for all $a\in A$ 
is effective as the \emph{Rickart mapping} on $A$. We note that, 
for $p\in P$, we have $p\,'=1-p=p\sp{\perp}$.  

Let $a,b\in A$. In \cite{SOUS} the notation $b\in CPC(a)$ means  
that, for all $p\in P$, $a\in C(p)\Rightarrow b\in C(p)$. Thus, 
$b\in CPC(a)\Leftrightarrow C(a)\cap P\subseteq C(b)$; hence, 
$CC(a)\subseteq CPC(a)$. For instance, by \cite[Lemmas 2.1 (vi) 
and 2.4 (iv)]{SOUS}, $a\dg, |a|, a\sp{+}\in CPC(a)$, but for our 
synaptic algebra $A$, we have the (possibly) stronger conditions 
$a\dg,\,|a|,\,a\sp{+}\in CC(a)$. 

In view of the remarks above, we can translate the results in 
\cite{SOUS} into our present formalism by replacing $a\,'$ by 
$1-a\dg$, $a\,''$ by $a\dg$, and $p\,'$ by $p\sp{\perp}=1-p$ 
for all $a\in A$ and all $p\in P$. Moreover, if $aCp$, we can 
replace $J\sb{p}(a)$ by $pa$ (or by $ap$). 
  
\begin{definition} \label{df:SpecRes}
Let $a\in A$ and $\lambda\in\reals$. Then:
\begin{enumerate}
\item The \emph{spectral lower and upper bounds} $L$ and $U$ for $a$ 
 are defined by $L :=\sup\{\lambda\in\reals:\lambda\leq a\}$ and 
 $U :=\inf\{\lambda\in\reals: a\leq\lambda\}$. 
\item The family of projections $(p\sb{\lambda})\sb{\lambda\in\reals}$ 
 defined by $p\sb{\lambda} :=1-((a-\lambda)\sp{+})\dg$ is called the 
 \emph{spectral resolution} of $a$. 
\item The family of projections $(d\sb{\lambda})\sb{\lambda\in\reals}$ 
 defined by $d\sb{\lambda} :=1-(a-\lambda)\dg$ is called the family 
 of \emph{eigenprojections} of $a$. If $d\sb{\lambda}\not=0$, then 
 $\lambda$ is called an \emph{eigenvalue} of $a$.
\end{enumerate}
\end{definition}

\begin{assumptions} \label{as:SRofa}
In what follows: $a\in A$; $L$ and $U$ are the spectral bounds 
for $a$; $(p\sb{\lambda})\sb{\lambda\in\reals}$ is the spectral 
resolution of $a$; and $(d\sb{\lambda})\sb{\lambda\in\reals}$ is 
the family of eigenprojections for $a$.
\end{assumptions}

By \cite[Theorem 3.1]{SOUS}, $-\infty<L\leq U<\infty$,  
$\|a\|=\max\{|L|,|U|\}$, and $L\leq a\leq U$. The following 
theorem is a consequence of \cite[Theorems 3.3, 3.5, and 3.6]{SOUS}. 

\begin{theorem}  \label{th:SpecProps}
For all $\lambda,\mu\in\reals$:
\begin{enumerate}
\item $p\sb{\lambda},\, d\sb{\lambda}\in CC(a)$; hence $p\sb{\lambda}
 Cp\sb{\mu}$, $p\sb{\lambda}Cd\sb{\mu}$, and $d\sb{\lambda}Cd\sb{\mu}$.

\item $p\sb{\lambda}(a-\lambda)\leq 0\leq(1-p\sb{\lambda})
 (a-\lambda)$. 

\item $\lambda\leq\mu\Rightarrow p\sb{\lambda}\leq p\sb{\mu}$ and
 $p\sb{\mu}-p\sb{\lambda}=p\sb{\mu}\wedge(1-p\sb{\lambda})$.

\item $\lambda<\mu\Rightarrow d\sb{\lambda}\leq p\sb{\lambda}\leq
 1-d\sb{\mu}\Rightarrow d\sb{\lambda}\perp d\sb{\mu}$.

\item $\mu\geq U\Leftrightarrow p\sb{\mu}=1$.

\item $\lambda<L\Rightarrow p\sb{\lambda}=0$, and $L<\lambda\Rightarrow
 0<p\sb{\lambda}$.

\item If $\alpha\in\reals$, then $p\sb{\alpha}=\bigwedge\{p\sb{\mu}:
\alpha<\mu\in\reals\}$.

\item If $\alpha\in\reals$, then $p\sb{\alpha}-d\sb{\alpha}=\bigvee 
\{p\sb{\lambda}: \alpha>\lambda\in\reals\}$.
\end{enumerate}
\end{theorem}

By \cite[Theorem 3.4, Remark 3.1, and Corollary 3.1]{SOUS}, we have 
the following theorem and corollary.

\begin{theorem} \label{th:SpecTh}
Suppose that
$\lambda\sb{0},\lambda\sb{1},...,\lambda\sb{n}\in\reals$ with
$\lambda\sb{0}<L<\lambda\sb{1}<\cdots<\lambda\sb{n}=U$, and let 
$\gamma\sb{i}\in\reals$ with $\lambda\sb{i-1}\leq\gamma\sb{i}
\leq\lambda\sb{i}$ for $i=1,2,...,n$.  Define $u\sb{i}:=p\sb{\lambda
\sb{i}}-p\sb{\lambda\sb{i-1}}$ for $i=1,2,...,n$, and let $\epsilon 
:=\max\{\lambda\sb{i}-\lambda\sb{i-1}: i=1,2,...,n\}$.
Then:
\[
 u\sb{1},u\sb{2},...,u\sb{n}\in P\cap CC(a),\ \
 \sum\sb{i=1}\sp{n}u\sb{i}=1,\text{\ and\ }\left\|a-\sum\sb{i=1}\sp{n}
 \gamma\sb{i}u\sb{i}\right\|\leq\epsilon.
\]
\end{theorem}

According to Theorem \ref{th:SpecTh}, $a$ can be written as 
a norm-convergent integral $a=\int\sb{L-0}\sp{U}\lambda\,dp
\sb{\lambda}$ of Riemann-Stieltjes type; hence $a$ not only 
determines, but it is determined by its spectral resolution.  

\begin{corollary} \label{co:AscendingLimit}
There exists an ascending sequence $a\sb{1}\leq a\sb{2}
\leq\cdots$ in $CC(a)$ such that each $a\sb{n}$ is 
a finite linear combination of projections in the family 
$(p\sb{\lambda})\sb{\lambda\in\reals}$ and $\lim\sb{n\rightarrow
\infty}a\sb{n}=a$. 
\end{corollary}

\begin{definition} \label{df:Spectrum}
A real number $\rho$ belongs to the \emph{resolvent set} of $a$ 
iff there is an open interval $I$ in $\reals$ with $\rho\in I$ 
such that $p\sb{\lambda}=p\sb{\rho}$ for all $\lambda\in I$. 
The \emph{spectrum} of $a$, in symbols $\opspec(a)$, is defined 
to be the complement in $\reals $ of the resolvent set of $a$. 
\end{definition}

As is proved in \cite{SOUS}, $\opspec(a)$ has all of the 
expected basic properties. For instance, by \cite[Theorem 4.3]
{SOUS}, $\opspec(a)$ is a closed nonempty subset of the closed 
interval $[L,U]\subseteq\reals$, $L=\inf(\opspec(a))\in\opspec(a)$, 
$U=\sup(\opspec(a))\in\opspec(a)$, and $\|a\|=\sup\{|\alpha|:
\alpha\in\opspec(a)\}$. By \cite[Theorem 4.4]{SOUS}, $a\in A\sp{+}
\Leftrightarrow\opspec(a)\subseteq\reals\sp{+}$, and by 
\cite[Corollary 5.1]{SOUS}, $a\in P\Leftrightarrow\opspec(a)
\subseteq\{0,1\}$. As a consequence of \cite[Theorem 4.2]{SOUS}, 
every isolated point of $\opspec(a)$ is an eigenvalue of $a$, 
and every eigenvalue of $a$ belongs to $\opspec(a)$. 

\begin{definition} \label{df:Simple}
An element in $A$ is \emph{simple} iff it is a finite linear 
combination of pairwise commuting projections.
\end{definition}

The following result is a consequence of \cite[Theorems 5.2 and 
5.3]{SOUS}.

\begin{theorem} \label{th:SimpleElements}
The simple elements of $A$ are precisely those with finite 
spectrum. Let $a$ be a simple element of $A$. Then $a$ 
can be written uniquely as $a=\sum\sb{i=1}\sp{n}
\alpha\sb{i}u\sb{i}$, where $\alpha\sb{1}<\alpha\sb{2}
<\cdots<\alpha\sb{n}$, $0\not=u\sb{i}\in P$, and 
$\sum\sb{i=1}\sp{n}u\sb{i}=1$. Moreover, $a$ is regular, 
$|a|=\sum\sb{i=1}\sp{n}|\alpha\sb{i}|u\sb{i}$, $\|a\|
=\max\{|\alpha\sb{i}|: i=1,2,...,n\}$, $a\dg=
\sum\sb{\alpha\sb{i}\not=0}u\sb{i}$, and $u\sb{i}=
d\sb{\alpha\sb{i}}$ for $i=1,2,...,n$. 
\end{theorem}

As a consequence of Corollary \ref{co:AscendingLimit}, each 
element $a\in A$ is the norm limit (hence by Theorem 
\ref{th:AscendingLimit} (ii) also the supremum) of an ascending 
sequence of pairwise commuting simple elements, and it follows 
that the simple elements in $A$ (hence by Theorem 
\ref{th:SimpleElements}, also the regular elements in $A$) are 
norm-dense in $A$. 

\begin{theorem} \label{th:SpectralCommutativity} 
If $b\in A$, then $bCa$ iff $bCp\sb{\lambda}$ for all $\lambda
\in\reals$.
\end{theorem} 
 
\begin{proof}
For $\lambda\in\reals$, we have $p\sb{\lambda}\in CC(a)$; hence 
$bCa$ implies that $bCp\sb{\lambda}$. Conversely, suppose that 
$bCp\sb{\lambda}$ for all $\lambda\in\reals$ and let $(a\sb{n})
\sb{n\in\Nat}$ be the ascending sequence in Corollary 
\ref{co:AscendingLimit}. As $a\sb{n}\in CC(a)$ for all $n\in\Nat$, 
the elements of the sequence $(a\sb{n})\sb{n\in\Nat}$ commute with 
each other. As each $a\sb{n}$ is a finite linear combination of 
projections $p\sb{\lambda}$, we have $a\sb{n}\in C(b)$ for all 
$n\in\Nat$, and it follows from [SA9] that $a\in C(b)$.  
\end{proof}

\begin{theorem} \label{th:CofaSynaptic}
$C(a)$ is norm-closed in $A$ and, with the partial order 
inherited from $A$, $C(a)$ is a synaptic algebra with unit $1$ 
and enveloping algebra $R$. Let $b,c\in C(a)$. Then: $b\circ c,
\,b\dg,\,|b|,\,b\sp{+},\,b\sp{-},\,J\sb{b}(c),\,\phi\sb{b}(c)
\in C(a)$; $0\leq b\Rightarrow b\sp{1/2}\in C(A)$; and the 
spectral resolution and family of eigenprojections of $b$ are 
the same whether calculated in $A$ or in $C(a)$.
\end{theorem}

\begin{proof}
Suppose that $(b\sb{n})\sb{n\in\Nat}$ is a sequence in C(a) and 
$b\sb{n}\rightarrow b\in A$. Then by Theorem 
\ref{th:SpectralCommutativity}, $b\sb{n}\in C(p\sb{\lambda})$ 
for all $n\in\Nat$ and all $\lambda\in\reals$, and it follows 
from Lemma \ref{lm:CpNormClosed} (ii) that $b\in C(p\sb{\lambda})$ 
for all $\lambda\in\reals$. Therefore, $b\in C(a)$ by Theorem 
\ref{th:SpectralCommutativity}, whence $C(a)$ is norm-closed in 
$A$.  The remainder of the proof is omitted as it is completely 
straightforward  
\end{proof}

\end{document}